\documentclass[11pt]{amsart}
\usepackage{amsfonts,amssymb,amsmath,amsthm}

\theoremstyle{plain}
\newtheorem{theorem}{Theorem}

\newtheorem{cor}{Corollary}

\newtheorem{prop}{Proposition}

\theoremstyle{definition}

\makeatletter

\@addtoreset{equation}{section}
\makeatother

\def\sts#1#2{\left\{#1\atop#2\right\}}

\def \Q{{\mathbb Q}}

\allowdisplaybreaks[4]

\author{Kalyan Chakraborty} 
\address{
Harish-Chandra Research Institute, Allahabad, 211019, Uttar Pradesh, India\\ 
Kerala school of Mathematics, Kunnamangalam, Kozhikode, 673571, Kerala, India
}
\email{kalyan@hri.res.in, kalychak@ksom.res.in
}

\author{Takao Komatsu}
\address{Department of Mathematical Sciences, School of Science, 
Zhejiang Sci-Tech University, 
Hangzhou 310018, 
China
}
\email{komatsu@zstu.edu.cn}


\title{Generalized hypergeometric Bernoulli numbers}

\begin{document}

\begin{abstract}
We introduce  generalized hypergeometric Bernoulli numbers for Dirichlet characters. We study their properties, including relations, expressions and determinants. At the end in Appendix we derive first few expressions of these numbers. 
\end{abstract}

\maketitle

\section{Introduction}\label{sect:intro} 

Let  ${}_1 F_1(a;b;z)$ denote the confluent hypergeometric function defined by 
$$
{}_1 F_1(a;b;z)=\sum_{n=0}^\infty\frac{(a)^{(n)}}{(b)^{(n)}}\frac{z^n}{n!}  
$$ 
with the rising factorial $(x)^{(n)}=x(x+1)\dots(x+n-1)$ ($n\ge 1$) and $(x)^{(0)}=1$. 
Also for $N\ge 1$, hypergeometric Bernoulli numbers $B_{N,n}$ (\cite{AK1, HN1,HN2,Ho1,Ho2,Kamano2,Ng}) are given by 
\begin{equation}  
\frac{1}{{}_1 F_1(1;N+1;t)}=\frac{t^N/N!}{e^t-\sum_{n=0}^{N-1}t^n/n!}=\sum_{n=0}^\infty B_{N,n}\frac{t^n}{n!} 
\label{def:hgb} 
\end{equation}   
with $B_{1,n}=B_n$, the classical Bernoulli numbers, defined by 
$$
\frac{t}{e^t-1}=\sum_{n=0}^\infty B_n\frac{t^n}{n!}\,. 
$$  
In \cite{HuKim} a result of Kamano on a multiple binomial sum is generalized, higher order hypergeometric Bernoulli polynomials are defined, and its differential equation and recursion formulas are studied among other properties. It is well-known that some of the known properties of the Bernoulli numbers can be obtained from fundamental relationships between complete and elementary symmetric functions. In \cite{HuKim18} two new closed form representations for the Apostol-Bernoulli polynomials are derived as natural generalizations of those of the Bernoulli polynomials.    
In \cite{Merca} an infinite family of relationships between complete and elementary symmetric functions is introduced. From these relationships, the author deduces connections between some lacunary recurrence relations with gaps of length $2 r$ for the Bernoulli numbers and the integer partitions into at most $r-1$ parts.  
In \cite{Zhu} new bounds for the ratio of two adjacent even-indexed Bernoulli numbers are obtained, solving Qi's conjecture on the related topic, getting a tighter lower bound for this ratio is obtained, and moreover, some conjectures on related topics are proposed. 

In \cite{Komatsu}, hypergeometric Cauchy numbers are defined and studied as analog of the hypergeometric Bernoulli numbers. This concept is further generalized in \cite{KY} by introducing the general hypergeometric Cauchy polynomials. 
In \cite{KZ}, the classical Euler numbers are generalized to define the higher order hypergeometric Euler numbers. 
Recently, in \cite{HuKomatsu}, by using the Hasse-Teichm\"uller derivatives, two explicit expressions for the related numbers of higher order Appell polynomials are obtained. One of them presents a determinant expression for the related numbers of higher order Appell polynomials, which involves several determinant expressions of special numbers, such as the higher order generalized hypergeometric Bernoulli and Cauchy numbers

In this paper, we introduce  generalized hypergeometric Bernoulli numbers for  Dirichlet characters. 
Let $\chi$ be a  Dirichlet character with conductor $f$ and $N$ be a positive integer.
We define the {\it generalized hypergeometric Bernoulli numbers} $B_{N,n,\chi}$ 
by
\begin{equation} 
\sum_{a=1}^f \frac{\chi (a) \ t^Ne^{at}/N!}{ e^{ft}-\sum_{n=0}^{N-1} \frac{(ft)^n}{n!} } =
\sum_{n=0}^{\infty} B_{N,n,\chi}\frac{t^n}{n!}\quad(|t|<2\pi/f)\,.
\label{def:gbn}
\end{equation} 
One recovers  the classical generalized Bernoulli number (see, e.g., \cite{Agoh,AIK,CE,W}) $B_{n,\chi}:=B_{1,n,\chi}$, which 
are simply Bernoulli numbers twisted by a Dirichlet character.  As seen in \cite[Chapter 4]{AIK},  generalized Bernoulli numbers and Bernoulli polynomials are related.   
A closed form expression for sums of products of any number of generalized Bernoulli numbers is given in \cite{CE} and
 shortened recurrence relations for generalized Bernoulli numbers and polynomials attached to a primitive Dirichlet character is being studied in \cite{Agoh}.
\section{Basic properties}
We begin with  some basic expressions of generalized hypergeometric Bernoulli numbers.
  \begin{prop} 
Let $\chi={\bf 1}$ be the trivial character of conductor $1$.
For any $N\geq 1$ and $n\geq 0$, put $N(n):=$min$\{ N-1,n\}$.
Then we have
\[ B_{N,n,{\bf 1}}=\begin{cases}
\displaystyle{ \sum_{m=0}^{N(n)} \binom{n}{m} B_{N,n-m}} & \text{if}\ n\ne N,\\
\\
\displaystyle{1+\sum_{m=0}^{N-1} \binom{N}{m} B_{N,N-m} }& \text{if}\ n= N.
\end{cases}
\]
Especially, for $N=1$, we have
\[ B_{1,n,{\bf 1}}=\begin{cases}
B_{n}& \text{if}\ n\ne 1,\\
\\
1/2& \text{if}\ n= 1.
\end{cases}
\]
\end{prop}
\begin{proof}
By the definition of generalized hypergeometric Bernoulli numbers, 
we have
\begin{align*}
\sum_{n=0}^{\infty} B_{N,n,{\bf 1}}\frac{t^n}{n!} & =\frac{ t^Ne^{t}/N!}{ e^{t}-\sum_{n=0}^{N-1} \frac{t^n}{n!} } \\
\\
& =\frac{t^N/N! (e^t-\sum_{n=0}^{N-1} \frac{t^n}{n!} )+t^N/N! \sum_{n=0}^{N-1} \frac{t^n}{n!}
}{ e^{t}-\sum_{n=0}^{N-1} \frac{t^n}{n!} } \\
\\
& =\frac{t^N}{N!} +\sum_{n=0}^{N-1} \frac{t^n}{n!} \times \sum_{n=0}^{\infty}  B_{N,n} \frac{t^n}{n!} \\
\\
& =\frac{t^N}{N!}+\sum_{m=0}^{N-1} \sum_{n=0}^{\infty} B_{N,n} \frac{t^{m+n}}{m!n!}\\
\\
& =\frac{t^N}{N!}+\sum_{\ell=0}^{N-1} \sum_{m=0}^{\ell } B_{N,\ell-m} \frac{t^{\ell}}{m!(\ell-m)!}
+ \sum_{\ell=N}^{\infty} \sum_{m=0}^{N-1 } B_{N,\ell-m} \frac{t^{\ell}}{m!(\ell-m)!}\\
\\
& =\sum_{\ell=0}^{N-1} \sum_{m=0}^{\ell }\binom{\ell}{m} B_{N,\ell-m} \frac{t^{\ell}}{\ell !}
+ \left( 1+\sum_{m=0}^{N-1} \binom{N}{m} B_{N,N-m} \right) \frac{t^N}{N!} \\
& +
\sum_{\ell=N+1}^{\infty} \sum_{m=0}^{N-1 } \binom{\ell}{m} B_{N,\ell-m} \frac{t^{\ell}}{\ell !},
\end{align*}
and we get the assertion.
\end{proof} 
Howard \cite{Ho1,Ho2} defined the hypergeometric Bernoulli polynomials $B_{N,n}(x)\in \Q[x]$ by
\begin{equation} 
\frac{t^N e^{x t}/N!}{e^t-\sum_{n=0}^{N-1} \frac{t^n}{n!} }=\sum_{n=0}^{\infty} B_{N,n}(x)\frac{t^n}{n!}.
\label{def:hyperberpol} 
\end{equation}  
Note that $B_n(x):=B_{1,n}(x)$ be the classical Bernoulli polynomial and $B_{N,n}(0)=B_{N,n}$.
For classical case, the following relation between the generalized Bernoulli numbers and the Bernoulli polynomials is well-known \cite[Proposition 4.1]{W}:
\[
B_{n,\chi}=f^{n-1} \sum_{a=1}^f \chi (a) B_n\left(\frac{a}{f} \right).
\]
The following proposition is a generalization of the above classical case. 

\begin{prop}\label{prop:GHBN-HBP}
Let $\chi$ be a  Dirichlet character of conductor $f$.
For any $N\geq 1$ and $n\geq 0$, 
\[
B_{N,n,\chi}=f^{n-N} \sum_{a=1}^f \chi (a) B_{N,n}\left(\frac{a}{f} \right).
\]
\end{prop}
\begin{proof}
By the definition (of  hypergeometric Bernoulli polynomial), we have,
\[ \sum_{n=0}^{\infty} B_{N,n}\left(\frac{a}{f}\right)\frac{(ft)^n}{n!}=\frac{(ft)^N e^{a/f \times ft}/N!}{e^{ft}-\sum_{n=0}^{N-1} \frac{(ft)^n}{n!} }.
\]
Now multiplying by $\chi (a)$ and taking a sum on $a$ from $1$ to $f$, we have
\begin{align*}
f^n\sum_{n=0}^{\infty} \sum_{a=1}^f \chi (a) B_{N,n} \left( \frac{a}{f} \right) \frac{t^n}{n!} & =
f^N \sum_{a=1}^f \frac{\chi (a) t^N e^{at}/N!}{e^{ft}-\sum_{n=0}^{N-1} \frac{(ft)^n}{n!} } \\
& \\
&=f^N\sum_{n=0}^{\infty} B_{N,n,\chi}\frac{t^n}{n!}.
\end{align*}
Therefore, 
\[
B_{N,n,\chi}=f^{n-N} \sum_{a=1}^f \chi (a) B_{N,n}\left(\frac{a}{f} \right)
\]
 for any $n\geq 0$.
\end{proof}

Consequently, we can express the generalized hypergeometric Bernoulli numbers in terms of hypergeometric Bernoulli numbers 

\begin{cor}  
For any integers $N\ge 1$ and $n\ge 0$,
\[
B_{N,n,\chi}=\sum_{a=1}^f\chi(a)\sum_{k=0}^n\binom{n}{k}B_{N,k}a^{n-k}f^{k-N}.
\] 
\label{cor10} 
\end{cor}  
\begin{proof}  
Since 
\begin{align*}  
\sum_{n=0}^\infty B_{N,n}(x)\frac{t^n}{n!}&=\frac{t^N e^{x t}/N!}{e^t-\sum_{n=0}^{N-1} \frac{t^n}{n!} }\\
&=\left(\sum_{n=0}^\infty B_{N,n}\frac{t^n}{n!}\right)\left(\sum_{m=0}^\infty x^m\frac{t^m}{m!}\right)\\
&=\sum_{n=0}^\infty\sum_{k=0}^n\binom{n}{k}B_{N,k} x^{n-k}\frac{t^n}{n!}, 
\end{align*}  
comparing the coefficients on both sides, we get 
\[
B_{N,n}(x)=\sum_{k=0}^n\binom{n}{k}B_{N,k} x^{n-k}.  
\]
By Proposition \ref{prop:GHBN-HBP} with $x=a/f$, we obtain the desired result.  
\end{proof} 

\section{Fundamental properties of the generalized hypergeometric Bernoulli polynomials} 

We also define the {\it generalized hypergeometric Bernoulli polynomials} $B_{N,n,\chi}$
by
\begin{equation} 
\sum_{a=1}^f \frac{\chi (a) \ t^N e^{(x+a)t}/N!}{e^{ft}-\sum_{n=0}^{N-1} \frac{(ft)^n}{n!} } =
\sum_{n=0}^{\infty} B_{N,n,\chi}(x)\frac{t^n}{n!}\quad(|t|<2\pi/f)\,.
\label{def:gbp}
\end{equation}
When $N=1$ in (\ref{def:gbp}), $B_{n,\chi}(x):=B_{1,n,\chi}(x)$ is the classical generalized Bernoulli polynomial.  
When $\chi={\bf 1}$, $B_{N,n}(x):=B_{N,n,{\bf 1}}(x)$ is the hypergeometric Bernoulli polynomial in (\ref{def:hyperberpol}).

\begin{prop} 
Let $\chi={\bf 1}$ be the trivial character of conductor $1$.
For any $N\geq 1$ and $n\geq 0$, we have
$$ 
B_{N,n,{\bf 1}}(x)=\begin{cases} 
\displaystyle{ \sum_{i=0}^n\binom{n}{i}B_{N,n-i}}(x) & \text{if}\ 0\leq n\leq N-1,\\
\\
\displaystyle{\binom{n}{N}x^{n-N}+\sum_{i=0}^{N-1}\binom{n}{i}B_{N,n-i}}(x)& \text{if}\ n\geq N.
\end{cases}
$$ 
Especially, for $N=1$, we have
$$ 
B_{1,n,{\bf 1}}(x)=\begin{cases}
1& \text{if}\ n=0,\\ 
n x^{n-1}+B_{n}(x)& \text{if}\ n\geq 1. 
\end{cases}
$$ 
\end{prop}
\begin{proof}  
By the definition in (\ref{def:gbp}), we have 
\begin{align*}  
&\sum_{n=0}^\infty B_{N,n,{\bf 1}}(x)\frac{t^n}{n!}=\frac{t^N e^{(x+1)t}/N!}{e^{t}-\sum_{n=0}^{N-1} \frac{t^n}{n!} }\\
&=\frac{t^N/N!\bigl(e^{t}-\sum_{n=0}^{N-1} \frac{t^n}{n!}\bigr)e^{x t}+t^N/N! e^{x t}\sum_{n=0}^{N-1} \frac{t^n}{n!}}{e^{t}-\sum_{n=0}^{N-1} \frac{t^n}{n!}}\\
&=\frac{t^N}{N!}e^{x t}+\left(\sum_{n=0}^{N-1} \frac{t^n}{n!}\right)\left(\sum_{n=0}^\infty B_{N,n}(x)\frac{t^n}{n!}\right)\\
&=\sum_{n=0}^\infty\frac{x^n}{N!n!}t^{N+n}+\sum_{n=0}^\infty\sum_{i=0}^{\min\{n,N-1\}}\binom{n}{i}B_{N,n-i}(x)\frac{t^n}{n!}\\
&=\sum_{n=N}^\infty\binom{n}{N}x^{n-N}\frac{t^n}{n!}+\sum_{n=0}^{N-1}\sum_{i=0}^n\binom{n}{i}B_{N,n-i}(x)\frac{t^n}{n!}\\
&\quad +\sum_{n=N}^\infty\sum_{i=0}^{N-1}\binom{n}{i}B_{N,n-i}(x)\frac{t^n}{n!}\,.  
\end{align*}
Comparing the coefficients on both sides, we get the result.  
\end{proof}

\begin{prop}
For $n\ge 0$,  
$$
B_{N,n,\chi}(x+y)=\sum_{k=0}^n\binom{n}{k}B_{N,k,\chi}(x)y^{n-k}\,. 
$$ 
\label{prop:xy}
\end{prop} 
\begin{proof}  
By the definition in (\ref{def:gbp}), we have 
\begin{align*}
\sum_{n=0}^\infty B_{N,n,\chi}(x+y)\frac{t^n}{n!}&=\left(\sum_{n=0}^\infty B_{N,n,\chi}(x)\frac{t^n}{n!}\right)e^{y t}\\
&=\left(\sum_{k=0}^\infty B_{N,k,\chi}(x)\frac{t^k}{k!}\right)\left(\sum_{l=0}^\infty\frac{y^l t^l}{l!}\right)\\
&=\sum_{n=0}^\infty\sum_{k=0}^n\binom{n}{k}B_{N,k,\chi}(x)y^{n-k}\frac{t^n}{n!}\,. 
\end{align*}
Comparing the coefficients on both sides, we get the result. 
\end{proof} 

When $x=0$ and $y$ is replaced by $x$ in Proposition \ref{prop:xy}, we have a relation between generalized hypergeometric Bernoulli polynomials and numbers.  

\begin{cor}
For $n\ge 0$,  
$$
B_{N,n,\chi}(x)=\sum_{k=0}^n\binom{n}{k}B_{N,k,\chi}x^{n-k}\,. 
$$ 
\label{prop:xy}
\end{cor}

There is another relation of generalized hypergeometric Bernoulli polynomials with their numbers in terms of Stirling numbers of the second kind $\sts{n}{k}$, whose generating function is given by 
$$
\frac{(e^t-1)^\ell}{\ell!}=\sum_{n=\ell}^\infty\sts{n}{\ell}\frac{t^n}{n!}\,.
$$   

\begin{prop}  
For $n\ge 0$,  
$$
B_{N,n,\chi}(x)=\sum_{k=0}^n\sum_{l=0}^k\binom{n}{k}\sts{k}{l}(x)_l B_{N,n-k,\chi}\,, 
$$ 
where $(x)_l=x(x-1)\cdots(x-l+1)$ ($l\ge 1$) denotes the falling factorial with $(x)_0=1$. 
\label{prop:st2}
\end{prop}  
\begin{proof}  
\begin{align*} 
\sum_{n=0}^\infty B_{N,n,\chi}(x)\frac{t^n}{n!}&=\left(\sum_{n=0}^\infty B_{N,n,\chi}\frac{t^n}{n!}\right)\bigl((e^t-1)+1\bigr)^x\\
&=\left(\sum_{n=0}^\infty B_{N,n,\chi}\frac{t^n}{n!}\right)\left(\sum_{l=0}^\infty\binom{x}{l}l!\sum_{n=l}^\infty\sts{n}{l}\frac{t^n}{n!}\right)\\
&=\left(\sum_{n=0}^\infty B_{N,n,\chi}\frac{t^n}{n!}\right)\left(\sum_{n=0}^\infty\sum_{l=0}^n\sts{n}{l}(x)_l\frac{t^n}{n!}\right)\\
&=\sum_{n=0}^\infty\sum_{k=0}^n\sum_{l=0}^k\binom{n}{k}\sts{k}{l}(x)_l B_{N,n-k,\chi}\frac{t^n}{n!}\,. 
\end{align*} 
Comparing the coefficients on both sides, we get the result. 
\end{proof}  

We can see that the generalized hypergeometric Bernoulli numbers are Appell.   

\begin{prop}  
For $n\ge 1$,  
$$
\frac{d}{d x}B_{N,n,\chi}(x)=n B_{N,n-1,\chi}(x)\,. 
$$ 
\label{appell}  
\end{prop}
\begin{proof}  
From the definition in (\ref{def:gbp}), 
\begin{align*} 
\frac{d}{d x}\sum_{n=0}^\infty B_{N,n,\chi}(x)\frac{t^n}{n!}&=t\sum_{n=0}^\infty B_{N,n,\chi}(x)\frac{t^n}{n!}\\
&=\sum_{n=0}^\infty(n+1)B_{N,n,\chi}(x)\frac{t^{n+1}}{(n+1)!}\\
&=\sum_{n=1}^\infty n B_{N,n-1,\chi}(x)\frac{t^{n}}{n!}\,. 
\end{align*}
Comparing the coefficients on both sides, we get the result. 
\end{proof}

\section{Recurrence relations}  

From now on, for convenience sake let us put
$$
S_n:=\sum_{a=1}^f\chi(a)a^n\quad\hbox{and}\quad S_n(x):=\sum_{a=1}^f\chi(a)(x+a)^n\,.
$$ 
The following are recurrence relations of the generalized hypergeometric Bernoulli numbers and polynomials with Dirichlet character.   

\begin{prop} 
For $n\ge 1$,  
\begin{align}  
\binom{N+n}{n}^{-1}\sum_{i=0}^n\binom{N+n}{i}B_{N,i,\chi}f^{N+n-i}=S_n\,, 
\label{rec1}\\
\binom{N+n}{n}^{-1}\sum_{i=0}^n\binom{N+n}{i}B_{N,i,\chi}(x)f^{N+n-i}=S_n(x)\,. 
\label{rec2}
\end{align}
\label{prp:brec}
\end{prop} 

\noindent 
{\it Remark.}  
When $N=1$ in Proposition \ref{prp:brec}, we have 
\begin{align*}  
\sum_{i=0}^{n-1}\binom{n}{i}\frac{B_{i+1,\chi}}{i+1}f^{n-i}+\frac{B_{0,\chi}}{n+1}f^{n+1}=S_n\,,\\
\sum_{i=0}^{n-1}\binom{n}{i}\frac{B_{i+1,\chi}(x)}{i+1}f^{n-i}+\frac{B_{0,\chi}(x)}{n+1}f^{n+1}=S_n(x)\,. 
\end{align*}
which are \cite[(1.12),(1.13)]{Agoh}. When $N=f=1$ in Proposition \ref{prp:brec}, we have the well-known basic recurrence relations 
\begin{align*}  
\sum_{i=0}^{n-1}\binom{n}{i}\frac{B_{i+1}}{i+1}+\frac{1}{n+1}=0\,;\\
\sum_{i=0}^{n-1}\binom{n}{i}\frac{B_{i+1}(x)}{i+1}+\frac{1}{n+1}=x^n\,. 
\end{align*}
When $\chi=\mathbf 1$ is the trivial character of conductor $1$, the relation (\ref{rec1}) is reduced to 
$$
\sum_{i=0}^n\binom{N+n}{i}B_{N,i,{\mathbf 1}}=0\,. 
$$ 
which is \cite[Proposition 1 (3)]{AK1}. 

\begin{proof}[Proof of Proposition \ref{prp:brec}.]  
We shall prove the relation (\ref{rec2}).  The relation (\ref{rec1}) is the special case when $x=0$.  From the definition in (\ref{def:gbp}), 
$$
\sum_{a=1}^f\chi(a)\frac{t^N}{N!}e^{(x+a)t}=\left(e^{ft}-\sum_{n=0}^{N-1} \frac{(ft)^n}{n!}\right)\left(
\sum_{n=0}^{\infty} B_{N,n,\chi}(x)\frac{t^n}{n!}\right)\,. 
$$ 
The right-hand side is equal to 
$$   
\left(\sum_{n=0}^\infty\frac{(ft)^{N+n}}{(N+n)!}\right)\left(
\sum_{n=0}^{\infty} B_{N,n,\chi}(x)\frac{t^n}{n!}\right)
=\sum_{n=0}^\infty\sum_{i=0}^n\frac{B_{N,i,\chi}(x)}{i!}\frac{f^{N+n-i}}{(N+n-i)!}t^{N+n}\,. 
$$   
The left-hand side is equal to 
$$   
\sum_{n=0}^\infty\sum_{a=1}^f\chi(a)\frac{t^N}{N!}(x+a)^n\frac{t^n}{n!}
=\sum_{n=0}^\infty\sum_{a=1}^f\chi(a)(x+a)^n\frac{t^{N+n}}{N!n!}\,. 
$$ 
Comparing the coefficients on both sides, we get 
$$
\sum_{i=0}^n\frac{B_{N,i,\chi}(x)}{i!}\frac{f^{N+n-i}}{(N+n-i)!}=\sum_{a=1}^f\frac{\chi(a)(x+a)^n}{N!n!}\,, 
$$ 
yielding the relation (\ref{rec2}).  
\end{proof}

\section{Expressions}

We give expressions of the generalized hypergeometric Bernoulli polynomials and numbers.  

\begin{theorem}  
For $n\ge 1$,  
\begin{align}  
B_{N,n,\chi}&=\sum_{k=0}^n\frac{k!}{f^{N-k}}\binom{n}{k}\sum_{r=0}^k T_r(k) S_{n-k}\,, 
\label{th:exp-n}\\ 
B_{N,n,\chi}(x)&=\sum_{k=0}^n\frac{k!}{f^{N-k}}\binom{n}{k}\sum_{r=0}^k T_r(k) S_{n-k}(x)\,. 
\label{th:exp-p}
\end{align}
where 
$$
T_r(k):=(-N!)^r\sum_{i_1+\cdots+i_r=k\atop i_1,\dots,i_r\ge 1}\frac{1}{(N+i_1)!\cdots(N+i_r)!}
$$ 
with $T_0(0)=1$ and $T_0(k)=0$ ($k\ge 1$).  
\label{th:express}
\end{theorem}

\noindent 
{\it Remark.}  
When $\chi=\mathbf 1$ is the trivial character of conductor $1$, expression (\ref{th:exp-n})  reduces to 
$$
B_{N,n,{\mathbf 1}}=n!\sum_{r=1}^n(-N!)^r\sum_{i_1+\cdots+i_r=n\atop i_1,\dots,i_r\ge 1}\frac{1}{(N+i_1)!\cdots(N+i_r)!}\,,
$$ 
which is \cite[Proposition 1 (4)]{AK1}. 

\begin{proof}[Proof of Theorem \ref{th:express}.]  
We shall prove  (\ref{th:exp-p}) as  (\ref{th:exp-n}) can be easily deduced.  
From the definition in (\ref{def:gbp}), 
\begin{align*}  
&\sum_{n=0}^\infty B_{N,n,\chi}(x)\frac{t^n}{n!}\\
&=\left(\sum_{a=1}^f\chi(a)\sum_{n=0}^\infty(x+a)^n\frac{t^n}{n!}\right)\frac{1}{f^N}\left(\sum_{n=0}^\infty\frac{N!f^n}{(N+n)!}t^n\right)^{-1}\\
&=\frac{1}{f^N}\left(\sum_{n=0}^\infty S_n(x)\frac{t^n}{n!}\right)\sum_{r=0}^\infty\left(-\sum_{n=1}^\infty\frac{N!f^n}{(N+n)!}t^n\right)^r\\ 
&=\frac{1}{f^N}\left(\sum_{n=0}^\infty S_n(x)\frac{t^n}{n!}\right)\sum_{r=0}^\infty(-1)^r\sum_{n=r}^\infty\sum_{i_1+\cdots+i_r=n\atop i_1,\dots,i_r\ge 1}\frac{(N!)^r f^n}{(N+i_1)!\cdots(N+i_r)!}t^n\\
&=\frac{1}{f^N}\left(\sum_{n=0}^\infty S_n(x)\frac{t^n}{n!}\right)\sum_{n=0}^\infty\sum_{r=0}^n(-N!)^r\sum_{i_1+\cdots+i_r=n\atop i_1,\dots,i_r\ge 1}\frac{n! f^n}{(N+i_1)!\cdots(N+i_r)!}\frac{t^n}{n!}\\
&=\frac{1}{f^N}\sum_{n=0}^\infty\sum_{k=0}^n\binom{n}{k}S_{n-k}(x)\sum_{r=0}^k(-N!)^r\sum_{i_1+\cdots+i_r=k\atop i_1,\dots,i_r\ge 1}\frac{k! f^k}{(N+i_1)!\cdots(N+i_r)!}\frac{t^n}{n!}\,.
\end{align*}
Comparing the coefficients on both sides, we get (\ref{th:exp-p}).  Note that the summation is empty when $k\ge 1$ and $r=0$.  
\end{proof} 

There are alternative expressions with additional binomial coefficients. 

\begin{theorem}  
For $n\ge 1$,  
\begin{align}  
B_{N,n,\chi}&=\sum_{k=0}^n\frac{k!}{f^{N-k}}\binom{n}{k}\sum_{r=0}^k\binom{k+1}{r+1}\widetilde{T}_r(k) S_{n-k}\,, 
\label{th:exp-n-2}\\ 
B_{N,n,\chi}(x)&=\sum_{k=0}^n\frac{k!}{f^{N-k}}\binom{n}{k}\sum_{r=0}^k\binom{k+1}{r+1}\widetilde{T}_r(k) S_{n-k}(x)\,. 
\label{th:exp-p-2}
\end{align}
where 
$$
\widetilde{T}_r(k):=(-N!)^r\sum_{i_1+\cdots+i_r=k\atop i_1,\dots,i_r\ge 0}\frac{1}{(N+i_1)!\cdots(N+i_r)!}
$$ 
with $\widetilde{T}_0(0)=1$ and $\widetilde{T}_0(k)=0$ ($k\ge 1$).  
\label{th:express-2}
\end{theorem}

\noindent 
{\it Remark.}  
When $\chi=\mathbf 1$ is the trivial character of conductor $1$,  (\ref{th:exp-n-2}) is reduced to 
$$
B_{N,n,{\mathbf 1}}=n!\sum_{r=1}^n\binom{n+1}{r+1}(-N!)^r\sum_{i_1+\cdots+i_r=n\atop i_1,\dots,i_r\ge 0}\frac{1}{(N+i_1)!\cdots(N+i_r)!}\,,
$$ 
which is \cite[Proposition 1 (5)]{AK1}. 

\begin{proof}[Proof of Theorem \ref{th:express-2}.]  
We shall prove the expression (\ref{th:exp-p-2}).  
Put  
$$ 
w:=\sum_{n=1}^\infty\frac{N! f^n}{(N+n)!}t^n\,. 
$$ 
Then, we have 
\begin{align*}  
B_{N,n,\chi}(x)&:=\left.\frac{d^n}{d t^n}\frac{1}{f^N}(1+w)^{-1}\sum_{a=1}^f\chi(a)e^{(x+a)t}\right|_{t=0}\\
&=\frac{1}{f^N}\sum_{k=0}^n\binom{n}{k}\left.\frac{d^k}{d t^k}\sum_{l=0}^\infty(-w)^l\right|_{t=0}\left.\frac{d^{n-k}}{d t^{n-k}}\sum_{\ell=0}^\infty S_\ell(x)\frac{t^\ell}{\ell!}\right|_{t=0}\\
&=\frac{1}{f^N}\sum_{k=0}^n\binom{n}{k}\sum_{l=0}^k\sum_{r=0}^l(-1)^r\binom{l}{r}\\
&\qquad\times\left.\frac{d^k}{d t^k}\left(\sum_{m=0}^\infty\frac{N! f^m}{(N+m)!}t^m\right)^r\right|_{t=0}\left.
\sum_{\ell=0}^\infty S_{n-k+\ell}(x)\frac{t^\ell}{\ell!}\right|_{t=0}\\
&=\frac{1}{f^N}\sum_{k=0}^n\binom{n}{k}\sum_{l=0}^k\sum_{r=0}^l(-1)^r\binom{l}{r}\\
&\qquad\times\left.\frac{d^k}{d t^k}\sum_{m=0}^\infty\sum_{i_1+\cdots+i_r=m\atop i_1,\dots,i_r\ge 0}\frac{(N!)^r f^m}{(N+i_1)!\cdots(N+i_r)!}t^m\right|_{t=0}
S_{n-k}(x)\\
&=\frac{1}{f^N}\sum_{k=0}^n\binom{n}{k}\sum_{l=0}^k\sum_{r=0}^l\binom{l}{r}\left.\sum_{m=k}^\infty\frac{m!f^m}{(m-k)!}\widetilde{T}_r(m)t^{m-k}\right|_{t=0}S_{n-k}(x)\\
&=\frac{1}{f^N}\sum_{k=0}^n\binom{n}{k}\sum_{l=0}^k\sum_{r=0}^l\binom{l}{r}k! f^k\widetilde{T}_r(k)S_{n-k}(x)\\
&=\sum_{k=0}^n\frac{k!}{f^{N-k}}\binom{n}{k}\sum_{r=0}^k\left(\sum_{l=r}^k\binom{l}{r}\right)\widetilde{T}_r(k)S_{n-k}(x)\\
&=\sum_{k=0}^n\frac{k!}{f^{N-k}}\binom{n}{k}\sum_{r=0}^k\binom{k+1}{r+1}\widetilde{T}_r(k)S_{n-k}(x)\,. 
\end{align*} 
\end{proof}

\section{Determinants}  

Recurrence formula from Proposition \ref{prp:brec}
  \begin{equation}  
B_{N,n,\chi}=\frac{S_n}{f^N}-\sum_{i=0}^{n-1}\frac{N!n!f^{n-i}}{(N+n-i)!i!}B_{N,i,\chi}
\label{rec:gbn}
\end{equation}
assists us to derive determinant expressions of generalized hypergeometric Bernoulli numbers and polynomials. Such determinant expressions were studied by Glaisher (\cite{Glaisher}) extensively for Bernoulli, Cauchy, Euler and more numbers.  

\begin{theorem}  
For $n\ge 1$,  
\begin{align}
B_{N,n,\chi}&=(-1)^n n!\left|\begin{array}{ccccc}  
\frac{N!}{(N+1)!}f&1&0&&\\  
\frac{N!}{(N+2)!}f^2&\frac{N!}{(N+1)!}f&&&\\
\vdots&&&&0\\
\frac{N!}{(N+n-1)!}f^{n-1}&\cdots&\frac{N!}{(N+2)!}f^2&\frac{N!}{(N+1)!}f&1\\
\widehat S_n&\widehat S_{n-1}&\cdots&\widehat S_2&\widehat S_1
\end{array}\right|\, 
\label{det:ghbn}
\end{align}
where 
$$
\widehat S_n=\frac{1}{f^N}\sum_{a=1}^f\chi(a)\left(\frac{N!}{(N+n)!}f^n-\frac{a^n}{n!}\right).
$$ 
Also
\begin{align}
B_{N,n,\chi}(x)&=(-1)^n n!\left|\begin{array}{ccccc}  
\frac{N!}{(N+1)!}f&1&0&&\\  
\frac{N!}{(N+2)!}f^2&\frac{N!}{(N+1)!}f&&&\\
\vdots&&&&0\\
\frac{N!}{(N+n-1)!}f^{n-1}&\cdots&\frac{N!}{(N+2)!}f^2&\frac{N!}{(N+1)!}f&1\\
\widehat S_n(x)&\widehat S_{n-1}(x)&\cdots&\widehat S_2(x)&\widehat S_1(x)
\end{array}\right|\,
\label{det:ghbp}
\end{align}
where 
$$ 
\widehat S_n(x)=\frac{1}{f^N}\sum_{a=1}^f\chi(a)\left(\frac{N!}{(N+n)!}f^n-\frac{(x+a)^n}{n!}\right)\,.
$$ 
\label{th:det}
\end{theorem} 

\noindent 
{\it Remark.}  
When $\chi=\mathbf 1$ is the trivial character of conductor $1$, by $\widehat S_n=N!/(N+n)!$, the determinant (\ref{det:ghbn}) is reduced to 
$$
B_{N,n,{\mathbf 1}}=(-1)^n n!\left|\begin{array}{ccccc}  
\frac{N!}{(N+1)!}&1&0&&\\  
\frac{N!}{(N+2)!}&\frac{N!}{(N+1)!}&&&\\
\vdots&&&&0\\
\frac{N!}{(N+n-1)!}&\cdots&\frac{N!}{(N+2)!}&\frac{N!}{(N+1)!}&1\\
\frac{N!}{(N+n)!}&\frac{N!}{(N+n-1)!}&\cdots&\frac{N!}{(N+2)!}&\frac{N!}{(N+1)!} 
\end{array}\right|\,  
$$ 
which is \cite[Theorem 1]{AK1}.

\begin{proof}[Proof of Theorem \ref{th:det}.]  
For simplicity, let
$$
\alpha_i=\frac{N!f^i}{(N+i)!}\quad(i\ge 1)\,,
$$ 
then 
$$
\widehat S_n=\alpha_n\frac{S_0}{f^N}-\frac{S_n}{f^N n!}=\frac{1}{f_N}\left(\alpha_n S_0-\frac{S_n}{n!}\right)\,.
$$ 
Furthermore, by putting $\beta_n=(-1)^n B_{N,n,\chi}/n!$, we shall show that 
\begin{align}  
\beta_n=\left|\begin{array}{ccccc}   
\alpha_1&1&0&&\\
\alpha_2&\alpha_1&&&\\
\vdots&&&&0\\ 
\alpha_{n-1}&\cdots&\alpha_2&\alpha_1&1\\
\widehat S_n&\widehat S_{n-1}&\cdots&\widehat S_2&\widehat S_1
\end{array}\right| 
\label{det:bb} 
\end{align} 
by using the recurrence relation 
\begin{equation}  
\beta_n=\frac{(-1)^{n}S_n}{f^N n!}-\sum_{i=0}^{n-1}(-1)^{n-i}\alpha_{n-i}\beta_i\quad(n\ge 1)\,. 
\label{recrel}
\end{equation} 
By definition, 
$$
\beta_0=B_{N,0,\chi}=\frac{S_0}{f^N}\,.  
$$ 
For $n=1$, by (\ref{recrel}),
$$
\beta_1=\alpha_1\beta_0-\frac{S_1}{f^N}=\widehat S_1\,.  
$$ 
Hence (\ref{det:bb}) is valid.  
Assume that (\ref{det:bb}) is valid up to $n-1$. Then by expanding the determinant along the first row repeatedly, the right-hand side of (\ref{det:bb}) is equal to 
\begin{align*}  
&\alpha_1\beta_{n-1}-\left|\begin{array}{ccccc}   
\alpha_2&1&0&&\\
\alpha_3&\alpha_1&&&\\
\vdots&\vdots&&&0\\ 
\alpha_{n-1}&\alpha_{n-3}&\cdots&\alpha_1&1\\
\widehat S_n&\widehat S_{n-2}&\cdots&\widehat S_2&\widehat S_1
\end{array}\right|\\
&=\alpha_1\beta_{n-1}-\alpha_2\beta_{n-2}+\cdots+(-1)^{n-1}\alpha_{n-2}\beta_2
 +(-1)^n\left|\begin{array}{cc} 
\alpha_{n-1}&1\\
\widehat S_n&\widehat S_1
\end{array}\right|\\
&=\alpha_1\beta_{n-1}-\alpha_2\beta_{n-2}+\cdots+(-1)^{n-1}\alpha_{n-2}\beta_2+(-1)^n\alpha_{n-1}\beta_1+(-1)^{n+1}\widehat S_n\\
&=\beta_n\,. 
\end{align*} 
The last equality is due to the recurrence relation (\ref{recrel}). Then the first relation (\ref{det:ghbn}) is proved.  
The second equation (\ref{det:ghbp}) can be proved similarly, by using the recurrence relation 
$$
\frac{B_{N,n,\chi}(x)}{n!}=\frac{S_n(x)}{f^N n!}-\sum_{i=0}^{n-1}\frac{N!f^{n-i}}{(N+n-i)!}\frac{B_{N,i,\chi}(x)}{i!}
$$ 
or 
$$   
\beta_n(x)=\frac{(-1)^n S_n(x)}{f^N n!}-\sum_{i=0}^{n-1}(-1)^{n-i}\alpha_{n-i}\beta_i(x)\quad(n\ge 1)\,. 
$$ 
\end{proof}

\section*{Acknowledgements} 

The authors are grateful to three anonymous referees for their precious comments and advices. This work has been partly done when the second author (T.K.) visited Harish-Chandra Research Institute in February 2020, where the first author (K. C.) was working before moving to his current institute. T. K. thanks K. C. for the warm hospitality of the institute. 
This work was initiated from a hint by Professor Miho Aoki of Shimane University. Both authors thank her.

\section*{Appendix}  

We derive first few values of $B_{N,n,\chi}$:
\begin{align*}  
B_{N,0,\chi}&=\frac{S_0}{f^N}\,,\\
B_{N,1,\chi}&=-\frac{S_0}{(N+1)f^{N-1}}+\frac{S_1}{f^N}\,,\\ 
B_{N,2,\chi}&=\frac{2 S_0}{(N+1)^2(N+2)f^{N-2}}-\frac{2 S_1}{(N+1)f^{N-1}}+\frac{S_2}{f^N}\,,\\ 
B_{N,3,\chi}&=\frac{3!(N-1)S_0}{(N+1)^3(N+2)(N+3)f^{N-3}}+\frac{3! S_1}{(N+1)^2(N+2)f^{N-2}}-\frac{3 S_2}{(N+1)f^{N-1}}+\frac{S_3}{f^N}\,,\\ 
B_{N,4,\chi}&=\frac{4!(N^3-N^2-6 N+2)S_0}{(N+1)^4(N+2)^2(N+3)(N+4)f^{N-4}}+\frac{4!(N-1)S_1}{(N+1)^3(N+2)(N+3)f^{N-3}}\\
&\qquad +\frac{12 S_2}{(N+1)^2(N+2)f^{N-2}}-\frac{4 S_3}{(N+1)f^{N-1}}+\frac{S_4}{f^N}\,. 
\end{align*}

Now 
\begin{align*}  
B_{N,0}&=1\,,\\
B_{N,1}&=-\frac{1}{N+1}\,,\\ 
B_{N,2}&=\frac{2}{(N+1)^2(N+2)}\,,\\
B_{N,3}&=\frac{3!(N-1)}{(N+1)^3(N+2)(N+3)}\,,\\
B_{N,4}&=\frac{4!(N^3-N^2-6 N+2)}{(N+1)^4(N+2)^2(N+3)(N+4)}\,. 
\end{align*}  
From  (\cite{AK1}) we can see that,
$$
B_{N,n,\chi}=\sum_{k=0}^n\binom{n}{k}\frac{S_{n-k}B_{N,k}}{f^{N-k}}\,, 
$$ 
as shown in Corollary \ref{cor10}.  

When $N=1$ and $\chi$ is not the trivial character, as $S_0=0$, we find that 
\begin{align*}   
B_{1,0,\chi}&=0\,,\\
B_{1,1,\chi}&=\frac{S_1}{f}\,,\\
B_{1,2,\chi}&=-S_1+\frac{S_2}{f}\,,\\
B_{1,3,\chi}&=\frac{f S_1}{2}-\frac{3 S_2}{2}+\frac{S_3}{f}\,,\\
B_{1,4,\chi}&=f S_2-2 S_3+\frac{S_4}{f}\,.
\end{align*} 

\end{document}